\documentclass[a4paper, 11pt, english]{article}
\usepackage[utf8]{inputenc}
\usepackage[T1]{fontenc}
\usepackage{babel}
\usepackage{graphicx}
\usepackage{stmaryrd}
\usepackage[a4paper]{geometry}
\usepackage{amsmath,amsfonts,amssymb,amsthm,epsfig,epstopdf,url,array}
\usepackage{rotating}
\usepackage[colorlinks=true,citecolor=red,linkcolor=blue,pdfpagetransition=Blinds]{hyperref}
\usepackage{cleveref}
\usepackage{nameref}
\usepackage{enumitem}
\usepackage{comment}
\Crefname{paragraph}{Section}{Sections}
\usepackage{fancyhdr}

\usepackage{fullpage}
\newcommand{\rank}{\operatorname{rank}}

\newcommand{\ensemblenombre}[1]{\mathbb{#1}}
\newcommand{\N}{\ensemblenombre{N}}

\newcommand{\R}{} 
\renewcommand{\R}{\ensemblenombre{R}}
\newcommand{\C}{\ensemblenombre{C}}

\renewcommand{\leq}{\leqslant}
\renewcommand{\geq}{\geqslant}

\newcommand{\norme}[1]{\left\lVert#1\right\rVert}

\newcommand{\dive}[1]{\mathrm{div}}

\theoremstyle{plain} 
\newtheorem{proposition}{Proposition}[section] 
\newtheorem{theorem}[proposition]{Theorem}

\newtheorem{lemma}[proposition]{Lemma}

\theoremstyle{definition}
\newtheorem{definition}[proposition]{Definition}
\newtheorem{rmk}[proposition]{Remark}

\numberwithin{equation}{section}

\makeatletter
\let\original@addcontentsline\addcontentsline
\newcommand{\dummy@addcontentsline}[3]{}
\newcommand{\DeactivateToc}{\let\addcontentsline\dummy@addcontentsline}
\newcommand{\ActivateToc}{\let\addcontentsline\original@addcontentsline}
\makeatother

\begin{document}

\title{The Kalman rank condition and the optimal cost for the null-controllability of coupled
Stokes systems\footnote{The research was done during a visit of the second author to La Sorbonne Universit\'e, funded by this university.}}
\author{Kévin Le Balc'h\footnote{Inria, Sorbonne Université, CNRS, Laboratoire Jacques-Louis Lions, Paris, France. {\tt kevin.le-balc-h@inria.fr}} , Luz de Teresa\footnote{Instituto de Matem\'aticas, Universidad Nacional Aut\'onoma de M\'exico, CDMX, M\'exico.  {\tt ldeteresa@im.unam.mx}}}

\maketitle
\begin{abstract}
This paper considers the controllability of a class of coupled Stokes systems with distributed controls. The coupling terms are of a different nature. The first coupling is through the principal part of the Stokes operator with a constant real-valued positive-definite matrix. The second one acts through zero-order terms with a constant real-valued matrix. We assume the controls have their support in different measurable subsets of the spatial domain. Our main result states that such a system is small-time null-controllable if and only if a Kalman rank condition is satisfied. Moreover, when this condition holds, we prove the sharp upper bound for the cost of null-controllability for these systems. Our method is based on two ingredients. We start from the recent spectral estimate for the Stokes operator from Chaves-Silva, Souza, and Zhang, and then we adapt Lissy and Zuazua’s strategy concerning the internal observability for coupled systems of linear parabolic equations to coupled Stokes systems.
\end{abstract}



\section{Introduction}

In the last fifteen years the challenging issue of controlling systems of coupled equations of the same nature has attracted the interest of the control community. This kind of problem is first motivated by practical applications because for instance coupled parabolic equations appear as model for chemical reactions, see \cite{ET89} or \cite{Pie10}. Moreover, it already comes naturally when studying insensitizing controls, that consists in finding a control function such that some functional depending on the solution of a partial differential equation is locally insensitive to perturbations of the initial condition, see   \cite{Lio92} where this notion was introduced. The first results dealing with the small-time null controllability for a single parabolic equation appear in the seminal papers of \cite{FR71}, \cite{LR95} and \cite{FI96} for the heat equation and in \cite{Ima01}, \cite{FCGIP04}, for the Stokes case (see also \cite{FCG06} and the references therein). That means that it is possible to drive an initial datum to zero at arbitrary time $T>0$ acting locally on a nonempty open subdomain $\omega$ of the whole spatial domain $\Omega$. The development in this area has been intense in the last years. In the case of scalar (heat) coupled equations an important number of challenging problems has been solved, see \cite{AKBGBdT11} for a survey of results until 2011. In the case of coupled Stokes or Navier-Stokes systems, to our knowledge, mainly the cases of two coupled systems have been treated \cite{Gue07}, \cite{CG14}, \cite{CGG15}. Very recently, the second author of the present paper and her collaborators considered the case of coupled Stokes equations and derive a necessary and sufficient condition for the small-time null-controllability to hold, see \cite{TdTWZ23} and \cite{TdTWZ24}. This type of condition is a generalization of the well-known Kalman rank condition in the finite dimensional case. The goal of this paper is to extend the results of \cite{TdTWZ23} and \cite{TdTWZ24} in several directions by using a complete different approach.

We consider a system coupling $n$ Stokes equations, $n \geq 1$, with $m$ distributed controls, $m \geq 1$, that is for $1 \leq i \leq n$,
\begin{equation}
\label{eq:controlcoupledstokes}
\left\{
\begin{array}{l l l}
\partial_t y^{(i)} + \sum_{j=1}^n d_{i,j} (-\Delta y^{(j)} + \nabla p^{(j)}) + \sum_{j=1}^n q_{i,j} y^{(j)} \\
\qquad =  \sum_{j=1}^m r_{i,j} v^{(j)} 1_{\omega_j} &\mathrm{in}\ (0,T)\times \Omega,\\
\nabla \cdot y^{(i)} = 0 &  \mathrm{in}\ (0,T)\times \Omega,\\
y^{(i)} = 0 &\mathrm{on}\ (0,T)\times \partial \Omega,\\
y^{(i)}(0,\cdot) = y_0^{(i)} &\mathrm{in}\  \Omega.
\end{array}
\right.
\end{equation}
Here, $\Omega$ is a smooth open bounded domain of $\R^N$, $N=2,3$ and $\omega_j$, $1 \leq j \leq m$, are measurable subsets contained in $\Omega$ such that their $N$-dimensional Lebesgue measure is strictly positive, i.e. $|\omega_j|>0$. For every $1 \leq i \leq n$, $y^{(i)}(t,x) \in \R^{N}$ and $p^{(i)} \in \R$ are the velocity and the pressure of a fluid. These Stokes equations are coupled through two types of coupling matrices, i.e. a coupling for the Stokes operator $D=(d_{i,j})_{1 \leq i, j \leq n} \in \R^{n \times n}$ and  a matrix $Q=(q_{i,j})_{1 \leq i, j \leq n} \in \R^{n \times n}$ coupling zero order terms. The interpretation of this model is that each component $k\in \{1,\dots, N\}$, $(y^{(i)}_k, \partial_{x_k}p^{(i)})$ satisfies the same coupled system of $n$ parabolic equations. That is, for every $1 \leq i \leq n$, for every $1 \leq k \leq N$, we have
\begin{equation}
\label{eq:controlcoupledstokesk}
\left\{
\begin{array}{l l l}
\partial_t y_k^{(i)} + \sum_{j=1}^n d_{i,j} (-\Delta y_k^{(j)} + \partial_{x_k}p^{(j)}) + \sum_{j=1}^n q_{i,j} y_k^{(j)} \\
\qquad =  \sum_{j=1}^m r_{i,j} v_k^{(j)} 1_{\omega_j} &\mathrm{in}\ (0,T)\times \Omega,\\
 y_k^{(i)} = 0 &\mathrm{on}\ (0,T)\times \partial \Omega,\\
y_k^{(i)}(0,\cdot) = y_{0,k}^{(i)} &\mathrm{in}\  \Omega.
\end{array}
\right.
\end{equation}
Note that the couplings, materialized by $D$ and $R$, are constant in time and space. The controls are denoted by $v^{(j)}(t,x) \in \R^{N}$ for $1 \leq j \leq m$ and act on the system through the control matrix $R=(r_{i,j})_{1 \leq i, j \leq n} \in \R^{n \times m}$, that is also constant in time and space.

We aim to state a Kalman rank condition ensuring the null-controllability of \eqref{eq:controlcoupledstokes} and to provide an optimal cost of the null-controllability for 
\eqref{eq:controlcoupledstokes} when the condition is fulfilled. 

First of all, we write \eqref{eq:controlcoupledstokes} in an abstract way. We introduce 
\begin{equation}
\mathcal{H} := \{ y \in L^2(\Omega)^N\ ;\ \nabla \cdot y = 0\ \in \Omega,\ y \cdot \nu = 0\ \text{on}\ \partial\Omega\},
\end{equation}
where denote by $\nu$ the unit outward normal vector field on $\partial \Omega$. The Stokes operator is defined as follows
\begin{equation}
\mathbb A := - \mathbb P \Delta,\ \mathcal{D}(\mathbb A) := \{ y \in [H^2(\Omega) \cap H_0^1(\Omega)]^N\ ;\ \nabla \cdot y = 0\ \in \Omega\},
\end{equation}
where $\mathbb P : L^2(\Omega)^N \to \mathcal H$ is the Leray projector, see \cite{Foiasetal01}. For $1 \leq j \leq m$, let us denote
\begin{equation}
\mathcal U_j:= L^2(\omega_j)^N,
\end{equation}
and the associated control operator $\mathbb B_j \in \mathcal{L}(\mathcal U_j, \mathcal H)$
\begin{equation}
\mathbb B_j v^{(j)} := \mathbb P \left(1_{\omega_j} \sum_{i=1}^N v_i^{(j)} e^{(i)}\right),\ \text{for}\ v^{(j)} = (v_1^{(j)}, \dots, v_N^{(j)}) \in \mathcal U_j
\end{equation}
{\color{black} where $e^{(i)}, i=1,\dots, N$ represents the canonical basis in $\R^N$.}
Let us finally denote 
\begin{equation}
\mathcal U = \mathcal U_1 \times \dots \times \mathcal U_m,
\end{equation}
and the full control operator $\mathbb B \in \mathcal L (\mathcal U, \mathcal H^m)$
\begin{equation}
\mathbb B v := \begin{pmatrix} \mathbb B_1 v^{(1)} \\ \cdots \\ \mathbb B_m v^{(m)}\end{pmatrix}\qquad \forall v = (v^{(1)}, \dots, v^{(m)}) \in \mathcal U.
\end{equation}
We assume the following condition on the matrix $D$, there exists $c>0$ such that
\begin{equation}
\label{eq:dpositive}
\langle D X, X \rangle_{\R^n} \geq c |X|^2,\qquad \forall X \in \R^{n \times n}.
\end{equation}

We can now consider the operator of the coupled system \eqref{eq:controlcoupledstokes}, 
\begin{equation}
\mathbb L := D \mathbb A + Q,\ \mathcal{D}(\mathbb L) := \mathcal{D}(\mathbb A)^n,
\end{equation}
that is for $y = (y^{(1)}, \dots, y^{(n)}) \in \mathcal{D}(\mathbb A)^n$,
\begin{equation}
\mathbb L y := \left(\sum_{j=1}^n d_{i,j} \mathbb A y^{(j)} + \sum_{j=1}^n q_{i,j} y^{(j)}\right)_{i \in \{1, \dots, n\}}.
\end{equation}
Denote $\mathcal U =\mathcal U_1 \times \dots \times \mathcal U_m$, the control operator is $R \mathbb B \in \mathcal L (\mathcal U, \mathcal H^n)$, for $v = (v^{(1)}, \dots, v^{(n)}) \in\mathcal U$,
\begin{equation}
R \mathbb B v := \left(\sum_{j=1}^m r_{i,j} \mathbb B_j y^{(j)} \right)_{i \in \{1, \dots, n\}}
\end{equation}
Then, the system \eqref{eq:controlcoupledstokes} is equivalent to
\begin{equation}
\label{eq:controlform}
y' + \mathbb{L} y = R \mathbb B v,\ y(0) = y_0,
\end{equation}
with
\begin{equation}
y_0 = (y_0^{(1)}, \dots, y_0^{(n)}),\quad y_0^{(i)}\in \R^N.
\end{equation}

We  introduce now the Kalman operator
\begin{equation}
\label{eq:Kalmanop}
\mathbb K : \mathcal H^m \times \mathcal{D}(\mathbb A)^m \times \dots \times \mathcal{D}(\mathbb A^{n-1})^m \to \mathcal{H}^n,\ w = (w^{(1)}, \dots, w^{(n)}) \mapsto \sum_{i=1}^n \mathbb L^{i-1} R w^{(i)}.
\end{equation}
In the above definition, note that $w^{(i)} = (w_1^{(i)}, \dots, w_m^{(i)}) \in \mathcal{D}(\mathbb A^{i-1})^m $ so that
\begin{equation}
R w^{(i)} = \left(\sum_{j=1}^m r_{k,j} w_j^{(i)}\right)_{k \in \{1, \dots, n\}} \in \mathcal{D}(\mathbb A^{i-1})^n = \mathcal{D}(\mathbb L^{i-1}),
\end{equation}
and thus the Kalman operator is well-defined. The adjoint of $\mathbb K$ is given by
\begin{equation}
\mathbb K^{*} \varphi = \left[R^{*} \varphi, R^{*} \mathbb L^{*} \varphi, \dots, R^{*} \mathbb{L^{*}}^{n-1} \varphi \right].
\end{equation}
\begin{definition}
Let $T>0$. The system \eqref{eq:controlform} is null-controllable at time $T>0$ if for any $y_0 \in \mathcal{H}^n$, there exists a control $v \in L^2(0,T;\mathcal U)^m$ such that the corresponding solution $y$ of \eqref{eq:controlform} satisfies $y(T, \cdot) = 0$.
\end{definition}

The main result is the following one.
\begin{theorem}
\label{th:mainresult1}
System \eqref{eq:controlform} is null-controllable at time $T>0$ if and only if 
\begin{equation}
\label{eq:kalman}
\ker \mathbb K^{*} = \{0\}.
\end{equation}
Moreover, if \eqref{eq:kalman} holds, then there exists $C=C(\Omega,\omega_1, \dots, \omega_m,D,Q,R)>0$ such that for every $T>0$, for any $y_0 \in \mathcal{H}^n$, there exists a control $v \in L^2(0,T;\mathcal U)$ satisfying
\begin{equation}
\label{eq:controlcost}
\norme{v}_{L^2(0,T;\mathcal U)} \leq C e^{C/T} \norme{y_0}_{ \mathcal{H}^n},
\end{equation}
such that the corresponding solution $y$ of \eqref{eq:controlform} satisfies $y(T, \cdot) = 0$.
\end{theorem}
Theorem \ref{th:mainresult1} is an extension of the results in \cite{TdTWZ24}  {\color{black} in the following sense:}
\begin{enumerate}
\item The coupling part on the Stokes operator, materialized by the matrix $D$, is not a priori diagonalizable and only satisfies the positive-definite assumption \eqref{eq:dpositive}. This contrasts with \cite{TdTWZ24} where the authors require that $D$ is a positive-definite diagonal matrix.
\item For $1 \leq j \leq m$, the controls $v^{(j)}$ are localized in different parts $\omega_j$ of the domain $\Omega$ and we only require that $|\omega_j| >0$. This type of assumption also differs from \cite{TdTWZ24} where the controls $v^{(j)}$ are all localized in the same nonempty open subset $\omega$.
\item The control cost estimate \eqref{eq:controlcost} is sharper than the one obtained in \cite{TdTWZ24} where one can deduce from the proof a control cost of the form $C e^{C/T^{11}}$. 
\item On the other hand, in our paper, we require all the  components of the controls, that is each $v^{(j)} \in \R^N$ for $1 \leq j \leq m$ while in \cite{TdTWZ24}, they assume that $v^{(j)} \in \R^{N-1} \times \{0\}$. One way to overcome this difficulty should be to get a spectral estimate for the Stokes operator with an observation on only the first $N-1$ components, see \Cref{rmk:openproblemstokesSpectral}.
\end{enumerate}

The necessity part of \Cref{th:mainresult1} is straightforward and can be reduced to a finite-dimensional argument. The sufficient part of \Cref{th:mainresult1} is more difficult. It is based on the spectral estimate for the Stokes operator of \cite{CSSZ20} and an adaptation of Lissy and Zuazua's strategy \cite{LZ19}   to coupled Stokes systems. Note that during the proof, we take special attention to the cost estimate of the control to get \eqref{eq:controlcost}.\medskip

The paper is organized as follows. In \Cref{sec:necessity} we prove the necessity part of \Cref{th:mainresult1} based on spectral properties of the Kalman operator \eqref{eq:Kalmanop} and we present some explicit examples when the Kalman condition \eqref{eq:kalman} holds. \Cref{sec:sufficiency} is devoted to prove  the sufficient part of \Cref{th:mainresult1} and to derive the optimal cost of null-controllability, i.e. \eqref{eq:controlcost}.\\

\textbf{Acknowledgements.} The firs author is partially supported by the Project TRECOS ANR-20-CE40-0009 funded by the ANR (2021--2024).

\section{Necessity of the Kalman rank condition}
\label{sec:necessity}

It is well-known that the Stokes operator $\mathbb A$ is a positive self-adjoint operator with compact resolvent in $\mathcal H$. In particular, its spectrum is discrete and composed of positive eigenvalues $(\gamma_p)_{p \geq 1}$. Let us consider an orthonormal basis $(\phi_p)_{p \geq 1}$ of $\mathcal H$ composed by eigenvectors of $\mathbb A$ associated with the eigenvalues $(\gamma_p)_{p \geq 1}$. For every $k \geq 1$, $\phi_k$ together with the associated pressure $p_k$ satisfy the following Stokes system
\begin{equation}
\label{eq:eigenfunctionstokes}
\left\{
\begin{array}{l l l}
- \Delta \phi_k + \nabla p_k = \gamma_k \phi_k &\mathrm{in}\ \Omega,\\
\nabla \cdot \phi_k = 0 &  \mathrm{in}\ \Omega,\\
\phi_k = 0 &\mathrm{on}\ \partial \Omega.
\end{array}
\right.
\end{equation}
We define for all $p \geq 1$ the matrix
\begin{equation}
\label{eq:Kp}
K_p := \left[ R | (\gamma_p D + Q) R | \cdots | (\gamma_p D + Q)^{n-1} R \right] \in \R^{n \times nm},
\end{equation}
and its adjoint
\begin{equation}
K_p^* = \begin{pmatrix} R^* \\  R^* (\gamma_p D^* + Q^*) \\ \vdots \\ R^* (\gamma_p D^* + Q^*)^{n-1} \end{pmatrix}.
\end{equation}
Then we have the following result that is a straightforward adaptation of \cite[Proposition 2.2]{AKBDGB09}   for coupled heat equations.
\begin{proposition}
\label{prop:equivalencekalman}
We have
\begin{equation}
\ker \mathbb K^* = \{0\} \Leftrightarrow \forall p \geq 1,\ \ker K_p^* = \{0\} \Leftrightarrow \forall p \geq 1,\ \mathrm{rank}\  K_p = n.
\end{equation}
\end{proposition}

In order to prove the necessity part of \Cref{th:mainresult1}, we consider the adjoint system of \eqref{eq:controlcoupledstokes}. That is, for $1 \leq i \leq n$, 
\begin{equation}
\label{eq:adjointcoupledstokes}
\left\{
\begin{array}{l l l}
- \partial_t \varphi^{(i)} + \sum_{j=1}^n d_{j,i}(- \Delta \varphi^{(j)} + \nabla \pi^{(i)} )+ \sum_{j=1}^n q_{j,i} \varphi^{(j)} = 0 &\mathrm{in}\ (0,T)\times \Omega,\\
\nabla \cdot \varphi^{(i)} = 0 &  \mathrm{in}\ (0,T)\times \Omega,\\
\varphi^{(i)} = 0 &\mathrm{on}\ (0,T)\times \partial \Omega,\\
\varphi^{(i)}(T,\cdot) = \varphi_0^{(i)} &\mathrm{in}\  \Omega.
\end{array}
\right.
\end{equation}
or in abstract way 
\begin{equation}
\label{eq:adjointform}
\varphi' + \mathcal{L}^* \varphi = 0,\ \varphi(T) = \varphi_0.
\end{equation}
The adjoint of the control operator is $\mathbb B^* R^* \in \mathcal{L}(\mathcal H^n,\mathcal U)$. More precisely, if $\varphi = (\varphi^{(1)}, \dots, \varphi^{(n)}) \in \mathcal H^n$ and if we denote by $\varphi_k^{(i)}$ the components of $\varphi^{(i)} \in \mathcal H$, then
\begin{multline}
\mathbb B^* R^* \varphi = \left((\mathbb B^* R^* \varphi)^{(1)}, \dots, (\mathbb B^* R^*\varphi)^{(m)}\right) \in \mathcal{U}, \\ (\mathbb B^* R^* \varphi)^{(i)} = \left(\sum_{j=1}^n r_{j,i} \sum_{k=1}^{N} \varphi_k^{(j)} e^{(k)} 1_{\omega_i} \right) \ i = 1, \dots, m.
\end{multline}

By the well-known Hilbert Uniqueness Method, see for instance \cite[Theorem 2.44]{Cor07}, the null-controllability of \eqref{eq:controlform} at time $T>0$ is equivalent to the following observability inequality. There exists $C>0$ such that for every $\varphi_0 \in \mathcal H^n$, the solution $\varphi$ of \eqref{eq:adjointcoupledstokes} satisfies
\begin{equation}
\label{eq:obs}
\int_{\Omega} |\varphi(0,x)|^2 dx \leq C \int_{0}^T \sum_{i=1}^m \int_{\omega_i} | (\mathbb B^* R^* \varphi)^{(i)} (t,x)|^2 dx dt.
\end{equation}
As a consequence, the sufficient part of our main result would be a consequence of the following result.
\begin{lemma}
If {\color{black}$\ker \mathbb K^* \not= \{0\} $, i.e.} \eqref{eq:kalman} does not hold, then there exists $\varphi_0 \in \mathcal H^n$,  such that the solution $\varphi$ of \eqref{eq:adjointcoupledstokes} satisfies $$\varphi(0, \cdot) \neq 0 \ \text{and}\ R^* \varphi = 0 \ \text{in}\ (0,T)\times\Omega.$$
\end{lemma}
\begin{proof}
Assume that {\color{black}$\ker \mathbb K^* \not= \{0\} $} then from \Cref{prop:equivalencekalman}, we obtain that there exist $p_0 \geq 1$ and $z_0 \in \ker K_{p_0}^* \subset \R^n$, $z_0 \neq 0$. Then, from the Cayley-Hamilton theorem, we deduce that
\begin{equation}
R^* e^{-(\gamma_{p_0} D^* + Q^*)(T-t) } z_0 = 0 \qquad t \in \R.
\end{equation}
For $z = (z^{(1)}, \dots, z^{(n)}) \in \R^n$ and $\phi \in \mathcal H$, we set
$$ z \phi := (z^{(1)}\phi, \dots, z^{(n)}\phi) \in \mathcal H^n.$$
{\color{black}Let $\phi_{p_0}$ be the corresponding eigenfunction with eigenvalue $\gamma_{p_0}$ to \eqref{eq:eigenfunctionstokes}}.
We can check that
$$ \varphi(t):= \left(e^{-(\gamma_{p_0} D^* + Q^*)(T-t)} z_0 \right) \phi_{p_0}\qquad t \geq 0,$$
 is a solution of \eqref{eq:adjointcoupledstokes} with $$\varphi(T) = z_0 \phi_{p_0} \neq 0.$$ Therefore, we have
 $$ \varphi(0) = \left(e^{-(\gamma_{p_0} D^* + Q^*)T} z_0 \right) \phi_{p_0} \neq 0\ \text{in}\ \Omega,$$
because  $e^{-(\gamma_{p_0} D^* + Q^*)T}$ is invertible, $z_0 \neq 0$ and $\phi_{p_0} \neq 0\ \text{in}\ \Omega$, and
$$ R^* \varphi(t) = \left(R^* e^{-(\gamma_{p_0} D^* + Q^*)(T-t)} z_0 \right) \phi_{p_0} = 0\ \text{in}\ (0,T)\times\Omega.$$
This leads to the conclusion of the proof.
\end{proof}

We finish this section by giving some examples where one can check explicitly if the Kalman rank condition \eqref{eq:kalman} holds. In this part, we will crucially use \Cref{prop:equivalencekalman}. Note that these examples are borrowed from \cite{TdTWZ24}, and $D$ is a diagonal matrix.\medskip

{\bf Examples.}

\noindent{\textbf{Case 1}:}
Assume  that $ D$ is diagonal and that all the viscosities are the same, that is, for all $i\in \{1,\ldots,n\}$, $d_{i,i}=d>0.
$
Then, we deduce from \cite{AKBDGB09} that \eqref{eq:kalman} is equivalent to
\begin{equation}\label{di11:51}
\rank \left[ R \mid QR \mid \cdots \mid Q^{n-1} R \right]=n,
\end{equation}
that is the original Kalman condition for the matrices $Q$ and $R$. In particular, 
\begin{equation*}
\left\{\begin{array}{lll}
	\displaystyle \partial_t y^{(i)}-d \Delta y^{(i)}+\nabla p^{(i)}
	+\sum_{j=1}^n q_{i,j} y^{(j)}
	=\sum_{j=1}^m r_{i,j} v^{(j)} 1_{\omega_i}  &\text{in} \ {(0,T)\times \Omega}, &(1\leq i\leq n)\\
	\nabla\cdot y^{(i)}=0 &\text{in }{(0,T)\times \Omega}, &(1\leq i\leq n)\\
	y^{(i)}=0 &\text{on }{(0,T)\times \partial\Omega}, &(1\leq i\leq n)\\
	y^{(i)}(0,\cdot)=y_0^{(i)} &\text{in }\Omega, &(1\leq i\leq n)
\end{array}\right.
\end{equation*}
is null-controllable if and only if the finite-dimensional linear system
\begin{equation*}
\left\{\begin{array}{ll}
	\displaystyle \frac{d}{dt} Y^{(i)}+\sum_{j=1}^n q_{i,j} Y^{(j)}
	=\sum_{j=1}^m r_{i,j} V^{(j)}  &\text{in} \ (0,T), \ (1\leq i\leq n)
	\\
	Y^{(i)}(0,\cdot)=Y_0^{(i)}\in \mathbb{R} & (1\leq i\leq n),
\end{array}\right.
\end{equation*}
is controllable.

\noindent{\textbf{Case 2}: simultaneous controllability.}
Assume that also $R$ is diagonal  and
$$
Q=0, \quad m=n, \quad R=(r_{i,i})_{i=1,\ldots,n}.
$$ 
To simplify notation we write $r_{i,i}:=r_i$, $d_{i,i}=d_i$.
Then, \eqref{eq:kalman} can be written as
$$
K_p:=\begin{bmatrix}
r_1  & \left(\gamma_p d_1\right)r_1 & \cdots & \left(\gamma_p d_1\right)^{n-1} r_1\\
r_2  & \left(\gamma_p d_2\right)r_2 & \cdots & \left(\gamma_p d_2\right)^{n-1} r_2\\
\vdots & \vdots & \ddots & \vdots \\
r_n  & \left(\gamma_p d_n\right)r_n & \cdots & \left(\gamma_p d_n\right)^{n-1} r_n
\end{bmatrix}
\in \mathcal{M}_{n}(\mathbb{R})
$$
In particular, using Vandermonde matrices, we see that \eqref{eq:kalman} is equivalent to
$$
\forall i\in \{1,\ldots,n\}, \ r_i\neq 0, 
\quad d_i\neq d_j \quad \text{if} \ i\neq j.
$$
In particular, taking $r_i=1$ for all $i$, we obtain that the system 
\begin{equation*}
\left\{\begin{array}{lll}
	\displaystyle \partial_t y^{(i)}-d_i \Delta y^{(i)}+\nabla p^{(i)}
	=v^{(i)} 1_{\omega}  &\text{in} \ {(0,T)\times \Omega}, &(1\leq i\leq n)\\
	\nabla\cdot y^{(i)}=0 &\text{in }{(0,T)\times \Omega}, &(1\leq i\leq n)\\
	y^{(i)}=0 &\text{on }{(0,T)\times \partial\Omega}, &(1\leq i\leq n)\\
	y^{(i)}(0,\cdot)=y_0^{(i)} &\text{in }\Omega, &(1\leq i\leq n)
\end{array}\right.
\end{equation*}
is null-controllable  provided all the viscosities $d_i$ are distinct.
The above system is {\it simultaneously} null-controllable in any time $T>0$.  Observe that in particular, since we are acting on each system we can take different control supports $\omega_i$ instead of $\omega$.

\noindent{\textbf{Case 3}: cascade systems.}  
Assume
$$
q_{i,j}=0 \quad \text{if} \ i\geq j+2, 
\quad
q_{i,i-1}\neq 0 \quad (2\leq i\leq n), \quad R=(\delta_{1,i})_{i=1,\ldots,n}
$$ 
where $\delta_{i,j}$ is the Kronecker delta.
Then a proof by induction {\color{black} on $n$ (the number of coupled equations)},  yields that $K_p$ is an upper triangular matrix of the form
$$
K_p:=\begin{bmatrix}
1  &  & &  & \\
0  & q_{2,1}& &  & \\
\vdots  & 0 & q_{3,2} q_{2,1}&   &\\
\vdots & \vdots & 0 &\ddots & \\
\vdots & \vdots & \vdots &  &\\
0  & 0 & 0 &\cdots   & \displaystyle\prod_{i=2}^n q_{i,i-1}
\end{bmatrix}
\in \mathcal{M}_{n}(\mathbb{R}),
$$
By hypothesis, the elements of diagonal are non zero.

\section{Sufficiency of the Kalman rank condition and optimal cost}
\label{sec:sufficiency}

The goal of this part is devoted to prove that if the Kalman rank condition \eqref{eq:kalman} holds then \eqref{eq:controlcoupledstokes} is null-controllable at every time $T>0$ and give an upper bound on the control cost.

\subsection{Well-posedness of the coupled Stokes system}

\begin{lemma}
The operator $D \mathbb A$ is a generator of an analytic semigroup on $\mathcal H^n$ that we denote by $(e^{t D \mathbb A})_{t \geq 0}$, satisfying for some $\lambda > 0$,
\begin{equation}
\norme{e^{t D \mathbb A}}_{\mathcal{L}( \mathcal H^n)} \leq e^{- \lambda t}\qquad \forall t \geq 0.
\end{equation}
\end{lemma}
\begin{proof}
Note that equivalently, $D \mathbb A$ can be defined by duality, by setting
\begin{equation}\label{eq:form0}
\left \langle D \mathbb A  \varphi , \psi \right\rangle_{\mathcal H}  =  D \int_\Omega \nabla \varphi:\nabla \psi\ dx \qquad \varphi, \psi \in \mathcal{D}( \mathbb A).
\end{equation}
By the positive-definite assumption on $D$, i.e. \eqref{eq:dpositive} and Poincaré's inequality, there exist $c_1,c_2 >0$ such that 
\begin{equation*}
\left \langle D \mathbb A  \varphi , \psi \right\rangle_{\mathcal H} 
\geq c_1 \norme{\nabla \varphi}_{L^2(\Omega)^N}^2 \geq c_2 \norme{\varphi}_{L^2(\Omega)^N}^2.
\end{equation*}
Thus the corresponding sesquilinear form associated to $D \mathbb A$ (this form is simply defined by the right hand side of \eqref{eq:form0}) is coercive on $\mathcal H$. We also readily check that this form is densely defined, closed and continuous on $\mathcal H$. By \cite[Theorem 1.52]{Ouh09},  $D \mathbb A$ generates an analytic semigroup on $\mathcal H$.
\end{proof}

By a perturbation argument (see, for instance, \cite[Section 3.2, Corollary 2.2]{Paz83}), we then obtain the following result.
\begin{lemma}
\label{lem:boundedsemigroup}
The operator $\mathbb L$ is a generator of an analytic semigroup on $\mathcal H^n$, denoted by $(e^{t \mathbb L})_{t \geq 0}$ satisfying for some $M >0$,
\begin{equation}
\norme{e^{t \mathbb L}}_{\mathcal{L}( \mathcal H^n)} \leq e^{M t}\qquad \forall t \geq 0.
\end{equation}
\end{lemma}

We now focus on the dissipation estimate. Recall that the eigenelements of the Stokes operator have been introduced in \eqref{eq:eigenfunctionstokes}. We consider the following orthogonal decomposition of $\mathcal H^n$, for some $\Gamma>0$,
\begin{equation}
\mathcal H^n = \mathcal P_{\Gamma} \bigoplus\mathcal P_{\Gamma}^{\perp},
\end{equation}
where
\begin{equation}
\mathcal P_{\Gamma} := \left\{ \left(\sum_{\gamma_k \leq \Gamma} a_k^{(i)} \phi_k\right)_{1 \leq i \leq n}\ ;\ a_k^{(i)} \in \R\quad \forall 1 \leq i \leq n\right\}.
\end{equation}
The projection on $ P_{\Gamma}$ is denoted by $\Pi_{\Gamma}$ that is
\begin{equation}
\Pi_{\Gamma} \left(\sum_{\gamma_k } a_k^{(i)} \phi_k \right) = \sum_{\gamma_k \leq \Gamma} a_k^{(i)} \phi_k\qquad \forall (a_k^{(i)})_{1 \leq i \leq n} \in \ell^2(\N^*)^n.
\end{equation}
We have the following result.
\begin{proposition}
\label{prop:dissipationestimate}
There exists $C>0$ such that for every $\Gamma >0$, $y_0 \in \mathcal P_{\Gamma}^{\perp}$, 
\begin{equation}
\label{eq:expdecay}
\norme{e^{t \mathbb L} y_0}_{\mathcal H^n} \leq C e^{- \Gamma t}\qquad \forall t \in [0,1].
\end{equation}
\end{proposition}
\begin{proof}
Let us take an initial data of the form
\begin{equation}
y_0^{(i)} = \sum_{\gamma_k > \Gamma} a_k^{(i)}(0) \phi_k.
\end{equation}
We have that the free evolution of the system gives
\begin{equation}
y^{(i)}(t,x) = \sum_{\gamma_k> \Gamma} a_k^{(i)}(t) \phi_k(x)\qquad (t, x) \in [0,1] \times \Omega,
\end{equation}
where
\begin{equation}
\label{eq:systedoDirect}
\left\{
\begin{array}{ll}
(a_k^{(i)})' +  \sum_{j=1}^n d_{ij} \gamma_k a_k^{(j)} + \sum_{j=1}^n q_{ij} a_k^{(j)} =0& \text{in}\ (0,1),\\
a_k^{(i)}(0)=a_{k,0}^{(i)}.&
\end{array}
\right.
\end{equation}
We take the scalar product of \eqref{eq:systedoDirect} with $a_k=(a_k^{(i)})_{1 \leq i \leq n}$ and we  obtain from \eqref{eq:dpositive} that
\begin{equation}
\|a_k(t)\|_{\R^n}^2 \leq C \exp(-\Gamma t) \exp(Q t) |a_k(0)|^2\qquad \forall t \in [0,1],\ \forall k \geq 1.
\end{equation}
By bounding $\|\exp(Qt)\| \leq C$ for $t \in [0,1]$, we obtain the result by using that
\begin{equation*}
\sum_{i=1}^n \norme{y^{(i)}(t,\cdot)}_{L^2(\Omega)^N}^2 =  \sum_{\gamma_k> \Gamma} |a_k(t)|^2 \leq C \exp(-\Gamma t) \sum_{\gamma_k> \Gamma} |a_k(0)|^2 \leq C  \exp(-\Gamma t) \norme{y(0,\cdot)}_{\mathcal H^n}^2.
\end{equation*}
This ends the proof.
\end{proof}

\subsection{A direct Lebeau-Robbiano strategy}

Our goal is to prove the sufficient part of \Cref{th:mainresult1} with the control cost estimate \eqref{eq:controlcost}. We fix $T \in (0,1)$.

Our starting point is the following spectral inequality, see \cite[Theorem 1.5]{CSSZ20} that is a generalization of \cite[Theorem 3.1]{CSL16} to the case of measurable subset.

\begin{proposition}
Let $\omega$ be a measurable subset contained in $\Omega$ such that $|\omega|>0$. There exists a positive constant $C>0$ such that for every sequence of complex numbers $(a_j) \in \C^{\N}$, for every $\Gamma >0$, we have
\begin{equation}
\label{eq:spectral}
\sum_{\gamma_k \leq \Gamma} |a_j|^2 = \int_{\Omega} \left| \sum_{\gamma_k \leq \Gamma} a_k \phi_k(x) \right|^2 dx \leq C e^{C \sqrt{\Gamma}} \int_{\omega} \left| \sum_{\gamma_k \leq \Gamma} a_k \phi_k(x) \right|^2 dx.
\end{equation}
\end{proposition}
\begin{rmk}
\label{rmk:openproblemstokesSpectral}
In order to generalize \Cref{th:mainresult1} to controls with only $N-1$ components, we ask the following question, as far as we know, still open. Does there exists a positive constant $C>0$ such that for every sequence of complex numbers $(a_j) \in \C^{\N}$, for every $\Gamma >0$, we have
\begin{equation}
\sum_{\gamma_k \leq \Gamma} |a_j|^2 = \int_{\Omega} \left| \sum_{\gamma_k \leq \Gamma} a_k \phi_k(x) \right|^2 dx \leq C e^{C \sqrt{\Gamma}} \sum_{i=1}^{N-1} \int_{\omega} \left| \sum_{\gamma_k \leq \Gamma} a_k \phi_k^{(i)}(x) \right|^2 dx,
\end{equation}
that is does the spectral estimate for the Stokes operator, with an observation on only the first $N-1$ components holds?
\end{rmk}

The following result is an observability of \eqref{eq:adjointcoupledstokes} for low frequences.
\begin{proposition}
\label{prop:obsLowfreq}
There exist $C>0$, $p_1 \in \N$ such that for every $\tau \in (0,T)$, $\Gamma>0$ and $\varphi_{0} \in \mathcal{P}_{\Gamma}$, the solution $\varphi$ of \eqref{eq:adjointcoupledstokes} satisfies
\begin{equation}
\label{eq:obslowFreq}
\norme{\varphi(\tau,\cdot)}_{\mathcal{H}^n}^2 \leq \frac{C}{\tau^{p_1}} e^{C \sqrt{\Gamma}} \int_{0}^{\tau} \sum_{i=1}^m \int_{\omega_i} | (\mathbb B^* R^* \varphi)^{(i)} (t,x)|^2 dx dt
\end{equation}
\end{proposition}

\begin{proof}

The proof is inspired by \cite[Section 3]{LZ19}.

We have for every $1 \leq i \leq n$,
\begin{equation*}
\label{eq:decompInit}
\varphi_{0}^{(i)}(x) = \sum_{\gamma_k \leq \Gamma} \varphi_{k,0}^{(i)} \phi_k(x)\qquad \varphi_{k,0}^{(i)} \in \R.
\end{equation*}

\indent Then, the solution $\varphi$ of \eqref{eq:adjointcoupledstokes} is given, component by component, as follows
\begin{equation}
\label{eq:decompsol}
\varphi^{(i)}(t,x) = \sum\limits_{\gamma_k \leq \Gamma} \varphi_{k}^{(i)}(t) \phi_k(x),
\end{equation}
where $\varphi = (\varphi_k^{(i)})_{1 \leq i \leq n}$ is the unique solution of the ordinary differential system for $1 \leq i \leq n$
\begin{equation}
\label{eq:systedo}
\left\{
\begin{array}{ll}
-(\varphi_k^{(i)})' + \gamma_k \sum_{j=1}^{n} d_{j,i} \varphi_{k}^{(j)} + \sum_{j=1}^{n} q_{j,i}\varphi_{k}^{(j)}= 0& \text{in}\ (0,\tau),\\
\varphi_k^{(i)}(0)=\varphi_{k,0}^{(i)}.&
\end{array}
\right.
\end{equation}

We apply the spectral estimate \eqref{eq:spectral} to get for every $1 \leq i \leq m$,
\begin{equation*}
\norme{\sum_{ \gamma_k \leq \Gamma} \left(\sum_{j=1}^n r_{ji} \varphi_k^{(j)}(t) \right) \phi_k(x)}_{L^2(\Omega)^N}^2 \leq C e^{C \sqrt{\Gamma}} \norme{\sum_{ \gamma_k \leq \Gamma} \left(\sum_{j=1}^n r_{ji} \varphi_k^{(j)}(t) \right) \phi_k(x)}_{L^2(\omega_i)^N}^2.
\end{equation*}
Then we sum on every $1 \leq i \leq m$ to get 
\begin{multline*}
\sum_{1 \leq i \leq m} \norme{\sum_{ \gamma_k \leq \Gamma} \left(\sum_{j=1}^n r_{ji} \varphi_k^{(j)}(t) \right) \phi_k(x)}_{L^2(\Omega)^N}^2 \\
\leq C e^{C \sqrt{\Gamma}} \sum_{1 \leq i \leq m} \norme{\sum_{ \gamma_k \leq \Gamma} \left(\sum_{j=1}^n r_{ji} \varphi_k^{(j)}(t) \right) \phi_k(x)}_{L^2(\omega_i)^N}^2
\end{multline*}
Then, we integrate between $0$ and $\tau$ to get
\begin{multline}
\label{eq:firstobslow}
\int_{0}^{\tau} \sum_{1 \leq i \leq m} \norme{\sum_{ \gamma_k \leq \Gamma} \left(\sum_{j=1}^n r_{ji} \varphi_k^{(j)}(t) \right) \phi_k(x)}_{L^2(\Omega)^N}^2 dt\\
 \leq C e^{C \sqrt{\Gamma}} \sum_{1 \leq i \leq n} \int_{0}^{\tau} \norme{\sum_{ \gamma_k \leq \Gamma} \left(\sum_{j=1}^m r_{ji} \varphi_k^{(j)}(t) \right) \phi_k(x)}_{L^2(\omega_i)^N}^2 dt.
\end{multline}
We recognize on the right hand side of \eqref{eq:firstobslow}
\begin{equation}
\label{eq:firstobslowright}
\sum_{1 \leq i \leq m} \int_{0}^{\tau} \norme{\sum_{ \gamma_k \leq \Gamma} \left(\sum_{j=1}^m r_{ji} \varphi_k^{(j)}(t) \right) \phi_k(x)}_{L^2(\omega)^N}^2 dt = \norme{R^* \varphi(t,x)}_{L^2(0,\tau;\mathcal U)}^2.
\end{equation}
So, now the goal is to give a lower bound of the left hand side of \eqref{eq:firstobslow} to obtain the desired observability inequality. By considering \eqref{eq:systedo}, we know from the Kalman condition that this system is observable for every $k$ so we have in particular
\begin{equation}
\label{eq:firstobslowleft}
\norme{\varphi_k(\tau)}_{\R^n} = \sum_{1 \leq i \leq n} |\varphi_k^{(i)}(\tau)|^2 \leq C(\tau,k,n) \int_{0}^{\tau} \left(\sum_{1 \leq i \leq m} \left(\sum_{j=1}^n r_{ji} \varphi_k^{(j)}(t) \right) \right)^2 dt.
\end{equation}
Moreover, from \cite[Lemma 3.1]{LZ19}, we have an estimate of the following form 
\begin{equation*}
C(\tau,k,n) \leq C \left( \sqrt{\tau} + \frac{1}{\tau^{n-1/2}}\right) F (\lambda_k),
\end{equation*}
where $C = C_n>0$ is independent of $\tau$ and $k$ and $F$ is a rational function of $\lambda_k$. We then deduce by using the orthogonality condition on the eigenfunctions of the Stokes operator in $L^2(\Omega)^N$ and \eqref{eq:firstobslow}, \eqref{eq:firstobslowright}, \eqref{eq:firstobslowleft} that 
\begin{equation*}
\norme{\varphi(\tau,\cdot)}_{\mathcal H^n}^2 =\norme{\sum_{\gamma_k \leq \Gamma} \varphi_k(\tau)}_{\mathcal H^n}^2 
\leq  C \max(C(\tau,k,n)) e^{C \sqrt{\Gamma}}\norme{R^* \varphi(t,x)}_{L^2(0,\tau;\mathcal U)}^2.
\end{equation*}
Threfore, we have \eqref{eq:obslowFreq} from the two previous estimates.
\end{proof}

By the Hilbert Uniqueness Method (see \cite[Theorem 2.44]{Cor07}) and the observability inequality for the solution of the adjoint system \eqref{eq:adjointcoupledstokes} for low frequences, see \Cref{prop:obsLowfreq}, we immediately deduce the following partial null-controllability result.
\begin{proposition}
\label{prop:nullcontrolLowfreq}
There exist $C>0$, $p_1 \in \N$ such that for every $\tau \in (0,T)$, $\Gamma>0$ and $y_{0} \in \mathcal{P}_{\Gamma}$, there exists $v \in L^2(0,\tau;\mathcal U)$ satisfying
\begin{equation}
\label{eq:coutpartialcontrol}
\norme{v}_{L^2(0,\tau;\mathcal U)} \leq \frac{C}{\tau^{p_1}} e^{C \sqrt{\Gamma}} \norme{y_0}_{\mathcal{H}^n},
\end{equation}
such that the corresponding solution $y$ of \eqref{eq:controlform} satisfies $y(T, \cdot) = 0$.
\end{proposition}

From \Cref{prop:nullcontrolLowfreq}, for every $\tau \in (0,T),\Gamma>0$ and $y_0 \in \mathcal P_{\Gamma}$, we introduce the notation:
\begin{equation}
\label{ControleMinL^2}
v_{\Gamma}(y_0, 0,\tau),
\end{equation}
such that the solution $y$ of \eqref{eq:controlform} satisfies $y(\tau,\cdot)=0$ and $v_{\Gamma}(y_0, 0,\tau)$ is the minimal-norm element of $L^2(0,\tau;\mathcal U)$ satisfying the estimate \eqref{eq:coutpartialcontrol}. In other words, $v_{\Gamma}(y_0, 0,\tau)$ is the projection of $0$ in the nonempty closed convex set of controls satisfying \eqref{eq:coutpartialcontrol} and driving the solution $y$ of \eqref{eq:controlform} in time $\tau$ to $0$.

We can now finish the proof of \Cref{th:mainresult1} by using the so-called Lebeau-Robbiano method \cite{LR95}.
\begin{proof}[Proof of \Cref{th:mainresult1}]
The proof is inspired by \cite[Section 6.2]{LRL12}. The constants $C,C'$ will increase from line to line.\\
\indent We split the interval $[0,T] = \cup_{k \in \N} [a_k, a_{k+1}]$ with $a_0 =0$, $a_{k+1} = a_{k} + 2 T_k$ and $T_k = \kappa T/2^{k}$ for $k\in \N$ and the constant $\kappa$ is chosen such that $2\sum\limits_{k=0}^{+\infty} T_k = T$. We also define $\mu_k = M2^{2k}$ for $M>0$ sufficiently large which will be defined later and for $k \in \N$. Then, we define the control $v$ in the following way:
\begin{itemize}
\item if $t \in (a_k, a_k+T_k)$, $v = v_{\mu_k}(\Pi_{\mathcal P_{\mu_k}}, a_k,T_k)$ and $y(t,\cdot) = e^{(t-a_k)\mathbb L}y(a_k,.) + \int_{a_k}^{t} e^{(t-s)\mathbb L} v(s,\cdot)ds$,
\item if $t \in (a_k+T_k, a_{k+1})$, $v = 0$ and $y(t,\cdot) = e^{(t-a_k-T_k)\mathbb L}y(a_k+T_k,\cdot)$,
\end{itemize}
In particular, we know that the semi-group $e^{t \mathbb L}$ is bounded for every $t \in [0,1]$ by \Cref{lem:boundedsemigroup}.\\
\indent By \eqref{eq:coutpartialcontrol}, the choice of $v$ during the interval time $[a_k, a_k +T_k]$ implies
\begin{align}
\label{eq:SolActiv}
\norme{y(a_k+T_k,\cdot)}_{L^2(\Omega)^n}^2 &\leq (C + C (\kappa 2^{-k} T)^{- p_1} e^{C \sqrt{M} 2^{k}}) \norme{y(a_k,\cdot)}_{L^2(\Omega)^n}^2\\
& \leq  \frac{C}{T^{p_1}} e^{C \sqrt{M}2^{k}} \norme{y(a_k,\cdot)}_{L^2(\Omega)^n}^2.\notag
\end{align}
During the passive period of the control, $t \in [a_k + T_k, a_{k+1}]$, the solution exponentially decreases by using the dissipation estimate \eqref{eq:expdecay} from \Cref{prop:dissipationestimate} 
\begin{equation}
\label{eq:SolPassiv}
\norme{y(a_{k+1},\cdot)}_{L^2(\Omega)^n}^2 \leq C' e^{-C'M2^{2k} T_k}  \norme{y(a_k+T_k,\cdot)}_{L^2(\Omega)^n}^2.
\end{equation}
Thus, by using $2^{2k} T_k = \kappa 2^{k}T$, \eqref{eq:SolActiv} and \eqref{eq:SolPassiv}, we have
$$\norme{y(a_{k+1},\cdot)}_{L^2(\Omega)^n}^2 \leq  \frac{C}{T^{p_1}} e^{C \sqrt{M} 2^{k}- C' M 2^{k}T }\norme{y(a_k,\cdot)}_{L^2(\Omega)^n}^2,$$
and consequently,
\begin{align}
\notag
\norme{y(a_{k+1},\cdot)}_{L^2(\Omega)^n}^2 &\leq  \left(\frac{C}{T^{p_1}}\right)^{k+1} e^{\sum_{j=0}^{k} \left(C\sqrt{M} 2^{j}-C'M T 2^{j}\right)} \norme{y_0}_{L^2(\Omega)^n}^2\\
& \leq e^{C/T + (C\sqrt{M} - C'MT)2^{k+1}}\norme{y_0}_{L^2(\Omega)^n}^2.\label{eq:estglobalsol}
\end{align} 
By taking $M$ such that $C \sqrt{M} - C' MT < 0$, for instance $M \geq 2(C/C'T)^{2}$, we conclude by \eqref{eq:estglobalsol} that we have $ \lim_{k \rightarrow + \infty} \norme{y(a_k,\cdot)} = 0$, i.e., $y(T,\cdot)=0$ because $t \mapsto y(t,\cdot) \in C([0,T];L^2(\Omega)^n)$ because $v \in L^2(Q_T)^m$ as we will show now.\\
\indent We have $\norme{v}_{L^2(0,\tau;\mathcal U)}^2 = \sum_{k=0}^{+\infty} \norme{v}_{L^2(a_k,a_k+T_k;\mathcal U)}^2$. Then, by using the estimate \eqref{eq:coutpartialcontrol} of the control on each time interval $(a_k,a_k+T_k)$ and the estimate \eqref{eq:estglobalsol}, we get
\begin{align}\notag
&\norme{v}_{L^2(0,\tau;\mathcal U)}^2\\
 &\leq \left(C T_0^{-p_1} e^{C\sqrt{M}} + \sum\limits_{k \geq 1} C T_k^{-p_1} e^{C\sqrt{M}2^k}  e^{C/T + (C\sqrt{M} - C'MT)2^{k}}\right) \norme{y_0}_{L^2(\Omega)^n}^2\notag\\
& \leq \left(C T^{-p_1} e^{C\sqrt{M}} + \sum\limits_{k \geq 1} C (2^{k}T^{-1})^{p_1}  e^{C/T} e^{(2C \sqrt{M}-C'MT) 2^{k}}\right) \norme{y_0}_{L^2(\Omega)^n}^2. \label{EsticontrolT}
\end{align}
By taking $M$ such that $2C \sqrt{M}-C'MT < 0$, for instance $M = 8 (C/C'T)^{2} \Rightarrow C \sqrt{M}-C'MT/2 = -C''/T$ with $C''>0$, we deduce from \eqref{EsticontrolT} that $v \in L^2(0,\tau;\mathcal U)$ and
\begin{align*}
\norme{v}_{L^2(0,\tau;\mathcal U)}^2&\leq C e^{C/T} \int_{0}^{+\infty} \left(\frac{\sigma}{T}\right)^{p_1} e^{-C'' \frac{\sigma}{T}} d\sigma \norme{y_0}_{L^2(\Omega)^n}^2\leq C e^{C/T} \norme{y_0}_{L^2(\Omega)^n}^2,
\end{align*}
which concludes the proof of \Cref{th:mainresult1}.
\end{proof}

\bibliographystyle{plain}
\small{\bibliography{stokes}}

\end{document}